\documentclass{amsart}
\usepackage{amsfonts}
\usepackage{amsmath,amscd}
\usepackage{amssymb}
\usepackage{amsthm}
\usepackage{newlfont}
\newcommand{\f}{\frac}

\newcommand{\ds}{\displaystyle}

 \newtheorem{thm}{Theorem}[section]
 
 \newtheorem{lem}[thm]{Lemma}
 
 \theoremstyle{definition}
 
 \theoremstyle{remark}

 \numberwithin{equation}{section}

\begin{document}

\title[Characterizing finite $p$-groups by their Schur multipliers]
 {Characterizing finite $p$-groups by their Schur multipliers}

\author[P. Niroomand]{Peyman Niroomand}
\address{School of Mathematics and Computer Science\\
Damghan University of Basic Sciences, Damghan, Iran}
\email{p$\_$niroomand@yahoo.com}






\begin{abstract}
It has been proved in \cite{ge} for every $p$-group of order $p^n$,
$|\mathcal{M}(G)|=p^{\f{1}{2}n(n-1)-t(G)}$, where $t(G)\geq 0$. In
\cite{be, el, zh}, the structure of $G$ has been characterized for
$t(G)=0,1,2,3$ by several authors. Also in \cite{sa}, the structure
of $G$ characterized when $t(G)=4$ and $Z(G)$ is elementary abelian.
This paper is devoted to classify the structure of $G$ when $t(G)=4$
without any condition.
\end{abstract}

\maketitle
\section{introduction and motivation}
The literature of $\mathcal{M}(G)$, the Schur multiplier is going
back to the work of Schur in 1904. It is important to know for which
classes of groups the structure of group can be completely described
only by the order of $\mathcal{M}(G)$. The answer to this question
for the class of finite $p$-group, was born in a result of Green. It
is shown that in \cite{ge}, for a given $p$-group of order $p^n$,
$|\mathcal{M}(G)|= p^{\f{1}{2}n(n-1)-t(G)}$ where $t(G)\geq0$.
Several authors tried to characterize the structure of $G$ by
$t(G)$. The structure of $G$ was classified in \cite{be, zh} for
$t(G)=0,1,2$. When $t(G)=3$, Ellis in \cite{el} classified the
structure  of $G$ by a different method to that of \cite{be, zh}. He
also could find the same results for $t(G)=0,1,2$.

By a similar technique to \cite[Theorem 1]{el}, the structure of
$p$-groups with $t(G)=4$ has been determined in \cite{sa} when
$Z(G)$ is elementary abelian, but it seems there are some missing
points in classifying the structure of these groups. The Main
Theorem  shows that there are some groups which are not seen in
these classification.

Recently in \cite{ni, ni2}, the author  gives some results on the
Schur multiplier of non-abelian $p$-groups. Handling these results,
the present paper  is devoted to classify the structure of all
finite $p$-groups when $t(G)=4$ without any condition.
\section{Some notations and known results}
In this section, we summarize some known results which are used
throughout this paper.

Using notations and terminology of \cite{el}, here $D_8$ and $Q_8$
denote the dihedral and quaternion group of order $8$, $E_1$ and
$E_2$ denote the extra special $p$-groups of order $p^3$ of exponent
$p$ and $p^2$, respectively. Also ${\mathbb{Z}}^{(m)}_{p^{n}}$
denotes the direct product of $m$ copies of the cyclic group of
order $p^n$.

In this paper, we say that $G$ has the property $t(G)=4$ or briefly
with $t(G)=4$, if $|\mathcal{M}(G)|=p^{\f{1}{2}n(n-1)-4}$.


\begin{thm}$\mathrm{(See}$ \cite[Main Theorem]{ni}$\mathrm{).}$\label{2} Let $G$ be a non-abelian
$p$-group of order $p^n$. If $|G^{'}| = p^k$, then we have
\[|\mathcal{M}(G)|\leq p^{\frac{1}{2}(n+k-2)(n-k-1)+1}.\] In
particular,
\[|\mathcal{M}(G)|\leq p^{\frac{1}{2}(n-1)(n-2)+1},\]
and the equality holds  in the last bound if and only if  $G=
E_1\times Z$, where $Z$ is an elementary abelian $p$-group.
\end{thm}


\begin{thm}$\mathrm{(See}$ \cite[Theorem 2.2.10]{kar}\label{3}$\mathrm{).}$ For every finite groups $H$ and $K$, we have
\[\mathcal{M}(H\times K)\cong\mathcal{M}(H)\times \mathcal{M}(K)\times \ds\frac{H}{H'}\otimes\ds\frac{K}{K'}.\]
\end{thm}

\begin{thm}$\mathrm{(See}$ \cite[Theorem 3.3.6]{kar}\label{4}$\mathrm{).}$
Let G be an extra special $p$-group of order $p^{2m+1}$.
Then\begin{itemize}
\item[(i)]  If $m\geq 2$, then ${|\mathcal M}(G)|=p^{2m^2-m-1}$.
\item[(ii)] If $m=1$, then the order of Schur multiplier of $D_8, Q_8, E_1$ and $E_2$ are
equal to  $2,1,p^2$ and $1$, respectively.
\end{itemize}
\end{thm}
\section{Main Result}
The aim of this section is to classify the structure of all $p$-groups when $t(G)=4$.
Since abelian groups with the property $t(G)=4$ are determined in \cite[Main Theorem (a)]{sa},
 we concentrate on non-abelian $p$-groups.
\begin{thm}\label{m2}Let $G$ be a non-abelian $p$-group of order $p^n$ and $n\geq 6$, then there is exactly
one group with the property $t(G)=4$ which is isomorphic to
$E_1\times {\mathbb{Z}}^{(3)}_p$.
\end{thm}
\begin{proof}First assume that $|G^{'}|=p$. By Theorem \ref{2}, if $G$ satisfies the condition of equality, then
$G\cong E_1\times Z$. One can check that by Theorems \ref{3} and
\ref{4}, $Z\cong {\mathbb{Z}}^{(3)}_p$. Otherwise,
$|\mathcal{M}(G)|=p^{\f{1}{2}n(n-1)-4}\leq p^{\f{1}{2}(n-1)(n-2)}$
so $n\leq 5$.

Now assume that $|G^{'}|=p^k (k\geq 2)$, Theorem \ref{2} implies
that
\[\f{1}{2}(n^2-n-8)\leq \frac{1}{2}(n+k-2)(n-k-1)+1\leq \f{1}{2}n(n-3)+1,\] and hence $n\leq
3$ unless $k=2$, in which case $n\leq 5$.
\end{proof}
The following theorem is a consequence of Theorems \ref{2} and
\cite[Main Theorem]{ni2}.

\begin{thm}\label{m3} Let $G$ be a non-abelian $p$-group  of order $p^5$ and $t(G)=4$. Then $G$ is isomorphic to the
\[{\mathbb{Z}}^{(4)}_p\rtimes_\theta \mathbb{Z}_p (p\neq2)~\text{or}~ D_8\times  {\mathbb{Z}}^{(2)}_p.\]
\end{thm}
Now we may assume that the order of all non-abelian groups with the
property $t(G)=4$ is exactly $p^4$, by using Theorems \ref{m2} and
\ref{m3}.

In the case $p=2$, the following lemma characterizes  all groups of
order $16$ with $t(G)=4$.
\begin{lem}\label{m4} Let $G$ be a $p$-group of order $16$ with $t(G)=4$, then $G$ is isomorphic
to one of the groups listed below
\begin{itemize}
\item[(i)]$Q_8\times\mathbb{Z}_2$,
\item[(ii)] $\langle a,b~|~a^4=1, b^4=1,[a,b,a]=[a,b,b]=1,[a,b]=a^2b^2\rangle$,
\item[(iii)]$\langle a,b,c~|~a^2=b^2=c^2=1, abc=bca=cab\rangle$.
\end{itemize}
\end{lem}
\begin{proof} The Schur multiplier of all groups of order 16 is determined in Table $I$ of \cite{br}
 $($also see \cite{ni1}$)$.
\end{proof}
\begin{lem}\label{m5} Let $G$ be a group of order $p^4 (p\neq2)$ and $Z(G)$ be of exponent $p^2$ with
$t(G)=4$.
 Then $G\cong E_4$, where $E_4$ is the unique central product of a cyclic group of order $p^2$ and a
non-abelian group of order $p^3$.
\end{lem}
\begin{proof}If  $G/G^{'}$ is not elementary abelian, then one can check that
$G$ is of exponent $p^3$, and so $|\mathcal{M}(G)|=1$. Thus
$G/G^{'}$ is elementary abelian, and hence that $G^{'}$ and Frattini
subgroup coincide. Using \cite[Corollary 2.5.3(i)]{kar}, we have
$|\mathcal{M}(G)|\geq p^2$. On the other hand, one can see that
$|\mathcal{M}(G)|\leq p^2$. The rest of proof is obtained directly
by using \cite[Lemma 2.1]{ni}.
 \end{proof}

\begin{lem}Let $G$ be a group of order $p^4 (p\neq2)$, $|G^{'}|=p$, $Z(G)$ of exponent $p$ and $t(G)=4$,
then $G$ is isomorphic to
\[E_2\times \mathbb{Z}_p~\text{or}~ \langle a,b~|~a^{p^2}=1, b^p=1,[a,b,a]=[a,b,b]=1\rangle.\]
\end{lem}
\begin{proof}
First suppose that $G/G^{'}$ is elementary abelian. Then \cite[Lemma
2.1]{ni} and Theorem \ref{4} follow that $G\cong
E_2\times\mathbb{Z}_p.$ Otherwise by  \cite[pp. 87-88]{bu}, there
are two groups
\[\langle a,b~|~a^{p^2}=1,b^p=1,[a,b,a]=[a,b,b]=1\rangle~\text{and}\]
\[ \langle a,b~|~a^{p^2}=b^{p^2}=1,[a,b,a]=[a,b,b]=1,[a,b]=a^p
\rangle\] such that $Z(G)\cong \mathbb{Z}_p\otimes \mathbb{Z}_p$,
$G/G^{'}\cong \mathbb{Z}_p\otimes \mathbb{Z}_{p^2}$ and  $G^{'}\cong
\mathbb{Z}_p$.

 Since the first has a central subgroup $H$ such that $G/H\cong E_1$, one can see
 that the order of its Schur multiplier is exactly $p^2$. On the other hand, \cite[Theorem 2.2.5]{kar} shows that
  the second group has
 $|\mathcal{M}(G)|=p$, which follows the result.
\end{proof}
\begin{lem}Let $G$ be a group of order $p^4 (p\neq2)$, $|G^{'}|=p^2$ and $t(G)=4$, then
$G$ is isomorphic to one of the following groups.
\begin{itemize}
\item[(i)]$\langle a,b|~a^9=b^3=1,[a,b,a]=1,[a,b,b]=a^6,[a,b,b,b]=1,\rangle$
\item[(ii)] $\langle a,b~|~a^p=1, b^p=1,[a,b,a]=[a,b,b,a]=[a,b,b,b]=1\rangle (p\neq3).$
\end{itemize}
\end{lem}
\begin{proof} The fifteen groups of odd order $p^4$ are listed in \cite{bu} or \cite{sch}. Our conditions reduce these groups
to the unique group (see also \cite[pp. 4177]{el} for more details).
\end{proof}
In the following Theorem we summarize the results.
\begin{thm}Let $G$ be a non-abelian group of order $p^n$ with $t(G)=4$,
then $G$ is isomorphic to one of the following groups.

For $p=2$,
\begin{itemize}
 \item[(1)]$D_8\times  {\mathbb{Z}}^{(2)}_p$,
 \item[(2)] $Q_8\times\mathbb{Z}_2$,
 \item[(3)] $\langle a,b~|~a^4=1, b^4=1,[a,b,a]=[a,b,b]=1,[a,b]=a^2b^2\rangle$
\item[(4)] $\langle a,b,c~|~a^2=b^2=c^2=1, abc=bca=cab\rangle$.
\end{itemize}
For $p\neq 2$,
\begin{itemize}
\item[(5)]$E_4$,
\item[(6)]$E_1\times {\mathbb{Z}}^{(3)}_p$,
\item[(7)]${\mathbb{Z}}^{(4)}_p\rtimes_\theta \mathbb{Z}_p,$
\item[(8)]$E_2\times {\mathbb{Z}}_p$,
\item[(9)]$\langle a,b~|~a^{p^2}=1, b^p=1,[a,b,a]=[a,b,b]=1\rangle$,
\item[(10)]$\langle
a,b~|~a^9=b^3=1,[a,b,a]=1,[a,b,b]=a^6,[a,b,b,b]=1\rangle$,
\item[(11)] $\langle a,b~|a^p=1, b^p=1,[a,b,a]=[a,b,b,a]=[a,b,b,b]=1\rangle (p\neq3)$.
\end{itemize}

\end{thm}

\end{document}